\documentclass[11pt]{amsart}

\usepackage{graphicx}
\usepackage[USenglish,english]{babel}
\usepackage[T1]{fontenc}
\usepackage[latin1]{inputenc}
\usepackage{amsfonts}
\usepackage{amssymb}
\usepackage{amsthm}
\usepackage{amsmath}
\usepackage{tikz}
\usetikzlibrary{decorations.pathreplacing}
\usepackage{float}
\usepackage{enumerate}
\usepackage{accents}
\usepackage{mathtools}
\usepackage{mathrsfs}
\usepackage{comment}
\usepackage{afterpage}
\usepackage{bbm}
\usepackage{dsfont}
\usepackage{stmaryrd}
\usepackage{hyperref}
\usepackage{mleftright}
\usepackage{subfig}
\usepackage{faktor}

\theoremstyle{plain}
\newtheorem{thm}{Theorem}[section]
\newtheorem{lem}[thm]{Lemma}
\newtheorem{prop}[thm]{Proposition}

\newtheorem*{thm*}{Main Theorem}
\newtheorem*{prop*}{Proposition}
\newtheorem*{cor*}{Corollary}
\newtheorem{thmintro}{Theorem}

\theoremstyle{definition}
\newtheorem{mydef}[thm]{Definition}

\newtheorem{rem}[thm]{Remark}
\newtheorem*{quest*}{Question}
\newtheorem{claim}{Claim}

\newtheorem{thmdef}[thm]{Theorem/Definition}

\DeclareMathOperator{\Ima}{Im}
\DeclareMathOperator{\diam}{diam}

\title{The Hausdorff- and Nagata-dimension of M\"obius spaces}

\author{Merlin Incerti-Medici}
\address{Institut f\"ur Mathematik, Universit\"at Z\"urich, Switzerland}
\email{merlin.incerti-medici@math.uzh.ch}

\begin{document}

\maketitle

\begin{abstract}
We study cross ratios from an axiomatic viewpoint and show that a space equipped with a cross ratios carries several notions of dimension. Specifically, we introduce notions of Hausdorff- and Nagata-dimension and prove that they are invariants of M\"obius spaces. This provides us with more direct methods of obtaining dimensions for boundaries of Gromov-hyperbolic spaces.
\end{abstract}

\tableofcontents




\section{Introduction} \label{sec:Introduction}

It is well known that for any $\mathrm{CAT(-1)}$ space $X$ we can define its boundary at infinity and equip it with a family of visual metrics $\{ \rho_x \}_{x \in X}$. It is also well known that these metrics are generally depending on the base point $x$. However, it is known that all these metrics belong to the same M\"obius structure, meaning that they all define the same cross-ratio-triple
\[ (w,x,y,z) \mapsto crt_{\rho}(wxyz) := (\rho(w,x)\rho(y,z) : \rho(w,y)\rho(x,z) : \rho(w,z)\rho(x,y) ) \in \mathbb{R}P^2.Ê\]

It turns out that M\"obius structures are a geometric structure in their own right that can exist without the aid of a metric. This has been studied in \cite{Buyalo} and \cite{IM17a}, the notion of M\"obius spaces having been introduced in the former. In fact, M\"obius spaces lie somewhere between metric spaces and topological spaces. They carry a natural topological structure, which is not necessarily metrizable. Thus, one may wonder which geometric notions that are known on metric spaces can be carried over to M\"obius spaces. In this paper, we will give a positive answer to this question in the case of the Hausdorff- and the Nagata-dimension, provided that the M\"obius structure is induced by a quasi-metric.\\

A quasi-metric on a set $Z$ is a map $\rho : Z \times Z \rightarrow \mathbb{R}$ that satisfies the same properties as a metric, except that the triangle-inequality is replaced by a weaker condition (see section \ref{sec:Preliminaries} for a precise definition). Furthermore, throughout the paper we will allow the possibility that the set $Z$ contains a point that `lies at infinity' with respect to $\rho$ (see section \ref{sec:Preliminaries}). Whenever $\rho$ has a point at infinity, we call $\rho$ an extended quasi-metric. Any (extended) quasi-metric induces a M\"obius structure via the formula above. If two quasi-metrics $\rho$, $\rho'$ induce the same M\"obius structure, we call them M\"obius equivalent. As we will see, any two M\"obius equivalent quasi-metrics can be obtained from each other through a finite number of rescalings and involutions (see section \ref{sec:Preliminaries}). Any quasi-metric notion that is invariant under rescaling and taking involution is therefore an invariant of the induced M\"obius structure.

There is an easy, naive way to generalise many notions from metric spaces to quasi-metric spaces, simply by replacing the metric by a quasi-metric in the definition. We will use this approach to generalise the notions of Hausdorff- and Nagata-dimension to quasi-metric spaces and prove that any two M\"obius equivalent quasi-metrics yield the same Hausdorff- and Nagata-dimension. In other words:

\begin{thmintro} \label{thmintro:HausdorffInvariance}
Let $(X,M)$ be a M\"obius space, $\rho$ a (possibly extended) quasi-metric inducing $M$. Let $\rho'$ be a rescaling or an involution of $\rho$.

Then $\dim_{Haus}(X,\rho) = \dim_{Haus}(X,\rho')$. In particular, we can define the Hausdorff-dimension of $(X,M)$ to be $\dim_{Haus}(X,M) := \dim_{Haus}(X,\rho)$. Furthermore, $\dim_{Haus}(X,\rho)$ is invariant under M\"obius equivalence.
\end{thmintro}

The Nagata-dimension has been introduced in a note by Assouad (cf. \cite{Assouad}). It is a varation of the asymptotic dimension and can be defined as follows:

\begin{mydef} \label{mydef:NagataDimension}
Let $(X,\rho)$ be a metric space and $\mathcal{B}$ a cover of $X$. The cover $\mathcal{B}$ is called {\it $C$-bounded} if for every $B \in \mathcal{B}$, diam$(B) \leq C$, i.\,e.\,the diameter of every set in the cover $\mathcal{B}$ is bounded by $C$.

Let $s > 0$, $m \in \mathbb{N}$. We say that a family of subsets $\mathcal{B} \subset \mathcal{P}(X)$ has {\it $s$-multiplicity $\leq m$} if for every set $U \in X$ with diam$(U) \leq s$, there are at most $m$ elements $B \in \mathcal{B}$ with $B \cap U \neq \emptyset$.

The Nagata-dimension $\dim_N(X,\rho)$ is defined to be the infimum of all $n$ such that there exists a constant $c > 0$ such that for all $s > 0$ there exists a $cs$-bounded cover of $X$ with $s$-multiplicity $\leq n+1$.
\end{mydef}

The class of metric spaces that have finite Nagata-dimension includes doubling spaces, metric trees, euclidean buildings, homogeneous or pinched negatively curved Hadamard manifolds and others (cf. \cite{LangSchlichenmaierNagata}). The Nagata-dimension is, however, not preserved by quasi-isometries.\\

As outlined earlier, we generalise the Nagata-dimension to quasi-metric spaces by replacing the metric in the definition by a quasi-metric. We will prove the following

\begin{thmintro} \label{thmintro:NagataInvariance}
Let $(X,\rho)$ be an (extended) quasi-metric space, $o \in X$. Let $\rho'$ be a rescaling or an involution of $\rho$. Then $\dim_N(X,\rho) = \dim_N(X,\rho')$.

In particular, given a M\"obius structure $M$ and any two quasi-metrics $\rho, \rho'$ that induce $M$, we have $\dim_N(X,\rho) = \dim_N(X,\rho')$. Thus we can define $\dim_N(X,M) := \dim_N(X,\rho)$ for any quasi-metric $\rho$ that induces $M$. Furthermore, $\dim_N(X,M)$ is invariant under M\"obius equivalence.
\end{thmintro}

These two theorems allow us to generalize the notions of Hausdorff- and Nagata-dimension from metric spaces to intrinsic M\"obius spaces. Even for boundaries of Gromov-hyperbolic spaces, which carry visual metrics $\rho_x$, this provides us with advantages. If $X$ is a $\delta$-hyperbolic space, then the construction of the visual metric $\rho_x$ relies on constructing a quasi-metric $\rho'_x = e^{-(\cdot \vert \cdot)_x}$ first, where $(Ê\cdot \vert \cdot)_x$ is the Gromov-product (see \cite{BS}, in particular section 3.3 for a systematic development of visual metrics). In order to obtain $\rho_x$, one has to take a sufficiently small power of $\rho'_x$ such that $\rho'^{\alpha}_x$ is bi-Lipschitz equivalent to a metric. Taking the exponential of a quasi-metric changes the Hausdorff-dimension though, as one can see when going through the definition. Indeed, taking the exponential corresponds to rescaling the interior metric space, changing the $\delta$ in the process. Therefore, developing notions of dimension for M\"obius spaces allows us to make finer distinctions between spaces.\\

The rest of the paper is organized as follows. In Section \ref{sec:Preliminaries}, we give precise definitions and review the basic properties of M\"obius structures needed. In Section \ref{sec:HausdorffDimension}, we will generalize the Hausdorff-dimension to quasi-metric spaces and prove Theorem \ref{thmintro:HausdorffInvariance}. We will use it as an illustrating example for our strategy. In Section \ref{sec:NagataDimension}, we will use the same strategy to prove Theorem \ref{thmintro:NagataInvariance}.\\

The results in section \ref{sec:NagataDimension} are generalizations of results that are known for metric spaces and are due to Lang-Schlichenmaier and Xiangdong respectively (cf.\,\cite{LangSchlichenmaierNagata} and \cite{XieXiangdong}). The upshot of the generalizations presented in this paper is that in considering not only metrics but also quasi-metrics, we can define the Hausdorff- and Nagata-dimension of intrinsic M\"obius spaces whose M\"obius structure is not induced by a metric.\\

{\bf Acknowledgments} The author is grateful to Viktor Schroeder for many inspiring discussions and helpful advice and to Urs Lang for helpful advice and bringing the results of Xie Xiangdong and Brodskiy, Dydak, Levin and Mitra to my attention.




\section{Preliminaries} \label{sec:Preliminaries}

Let $Z$ be a set, $\rho : Z \times Z \rightarrow \mathbb{R}$ a map and $K \geq 1$. We say that $\rho$ is a {\it semi-metric} if it is symmetric, non-negative and for all $z$, $z' \in Z$, we have that $\rho(z,z') = 0$ if and only if $z = z'$. We say that $\rho$ is a {\it $K$-quasi-metric} if it is a semi-metric and for all $x, y, z \in Z$, we have $\rho(x,z) \leq K \max(\rho(x,y), \rho(y,z))$. Generalizing the definition of a quasi-metric, we say that $\rho : Z \times Z \rightarrow [0, \infty]$ is an {\it extended $K$-quasi-metric} if there exists exactly one point $\omega \in Z$, such that for all $x \in Z \diagdown \{ \omega \}$, $\rho(x, \omega) = \infty$, $\rho( \omega, \omega) = 0$ and the restriction of $\rho$ to $Z \diagdown \{ \omega \} \times Z \diagdown \{ \omega \}$ is a $K$-quasi-metric. We call $\omega$ the {\it point at infinity} with respect to $\rho$. A motivating example for this notion is the Riemannian sphere, seen as the union $\mathbb{C} \cup \{ \infty \}$. Analogously, we define the notions of extended metrics and extended semi-metrics. Throughout this paper, we use the convention that `infinite distances cancel', i.e.\,if $\omega$ lies at infinity with respect to $\rho$, then for all $z$, $z' \in Z \setminus \{ \omega \}$, we have $\frac{ \rho(z, \omega) }{ \rho( z', \omega)} = 1$.

We call an $n$-tuple $(z_1, \dots, z_n) \in Z^n$ {\it non-degenerate} if and only if for all $i \neq j$, we have $z_i \neq z_j$.

To define M\"obius structures, we need the notion of admissible quadruples.
\begin{mydef}
A quadruple $(z_1, z_2, z_3, z_4) \in Z^4$ is admissible if there exists no triple $i \neq j \neq k \neq i$ such that $z_i = z_j = z_k$. We denote the set of admissible quadruples by $\mathcal{A}$.
\end{mydef}

Consider the ordered triple $((12)(34), (13)(42), (14)(23))$. The symmetric group of four elements $\mathcal{S}_4$ acts on this triple by permuting the numbers 1-4. Whenever $\sigma \in \mathcal{S}_4$ acts on the numbers, it induces a permutation on the triple. Define $\varphi(\sigma) \in \mathcal{S}_3$ to be the permutation on the triple induced by the action of $\sigma$. It is easy to check that $\varphi : \mathcal{S}_4 \rightarrow \mathcal{S}_3$ is a group homomorphism. One can interpret the expression $(12)(34)$ to denote two opposite edges of a tetrahedron whose corners are labeled by the numbers 1-4. In this interpretation, $\varphi$ is the group homomorphism that sends a permutation of the corners to the induced permutation of pairs of opposite edges.\\

A M\"obius structure will be a map that sends admissible quadruples to particular triples of real numbers, such that certain symmetries are satisfied. The image of a M\"obius structure will be contained in the set
\[Ê\overline{L_4} := \{ (x,y,z) \in \mathbb{R}^3 \vert x + y + z = 0 \} \cup \{ (0, \infty, -\infty), (-\infty, 0, \infty), (\infty, - \infty, 0) \}. \]

\begin{mydef} \label{def:GeneralizedMobiusStructure}

Let $Z$ be a set with at least three points. A map $M : \mathcal{A} \rightarrow \overline{L_4}$ is called a {\it M\"obius structure} if and only if it satisfies the following conditions:

\begin{enumerate}

\item[1)] For all $P \in \mathcal{A}$ and all $\pi \in \mathcal{S}_4$, we have

\[ M(\pi P) = sgn(\pi)\varphi(\pi)M(P). \]

\item[2)] For $P \in \mathcal{A}$, $M(P) \in L_4$ if and only if $P$ is non-degenerate.

\item[3)] For $P = (x,x,y,z)$, we have $M(P) = (0, \infty, -\infty)$.

\item[4)] Let $(x,y, \omega, \alpha, \beta)$ be an admissible $5$-tuple $(x, y, \omega, \alpha, \beta)$ such that $(\omega, \alpha, \beta)$ is a non-degenerate triple, $\alpha \neq x \neq \beta$ and $\alpha \neq y \neq \beta$. Then, there exists some $\lambda = \lambda(x,y,\omega,\alpha,\beta) \in \mathbb{R} \cup \{ \pm \infty \}$ such that

\[ M(\alpha x \omega \beta) + M(\alpha \omega y \beta) - M(\alpha x y \beta) = (\lambda, -\lambda, 0). \]

Moreover, when $(\omega, \alpha, \beta)$ is non-degenerate, $x \neq \beta$ and $y \neq \alpha$, the first component of the left-hand-side expression is well-defined. Analogously, the second component of the left-hand-side expression is well-defined when $(\omega, \alpha, \beta)$ is non-degenerate, $x \neq \alpha$ and $y \neq \beta$.

\end{enumerate}

The pair $(Z,M)$ is called a {\it M\"obius space}.

\end{mydef}

Given a M\"obius structure $M$, we write $M = (a,b,c)$, where $a$, $b$, $c$ are the components of $M$. When studying M\"obius structures, the following two maps, which are both equivalent to $M$, can be useful. We define the {\it cross ratio}
\begin{equation*}
\begin{split}
cr : \mathcal{A} & \rightarrow [0, \infty]\\
cr(w,x,y,z)& := e^{c(w,x,y,z)}.
\end{split}
\end{equation*}
Denoting
\[ \overline{\Delta} = \{ (\alpha : \beta : \gamma) \in \mathbb{R}P^2 \vert \alpha, \beta, \gamma > 0 \} \cup \{ (0:1:1) , (1:0:1), (1:1:0) \}, \]
\[ a := a(w,x,y,z) \qquad b := b(w,x,y,z) \qquad c := c(w,x,y,z), \]
we define the {\it cross ratio-triple}
\begin{equation*}
\begin{split}
crt : \mathcal{A} & \rightarrow \overline{\Delta}\\
crt(w,x,y,z) := & ( e^{ \frac{1}{3} (c - b) } : e^{ \frac{1}{3} (a - c)Ê} : e^{ \frac{1}{3} (b - a ) } ).
\end{split}
\end{equation*}

The map $M$ can be recovered from both $cr$ and $crt$ (note that $(a,b,c) \in \overline{L_4}$ and thus $a + b + c = 0$). Abusing notation, we will also refer to $(Z, cr)$ and $(Z, crt)$ as M\"obius spaces.

In \cite{Buyalo}, Buyalo showed that for every M\"obius structure $M$, there exists a (possibly extended) semi-metric $\rho$, such that for every admissible quadruple $(w,x,y,z)$, we have
\[ cr(w,x,y,z) = \frac{\rho(w,x) \rho(y,z)}{ \rho(w,y) \rho(x,z)}. \]
For this reason, M\"obius structures are also thought of as generalisations of cross ratios. Given a semi-metric $\rho$, we speak of the M\"obius structure or cross ratio induced by $\rho$. When two different semi-metrics $\rho$, $\rho'$ induce the same M\"obius structure, we call them M\"obius equivalent. A map $f : (Z, M) \rightarrow (Z', M')$ between two M\"obius spaces is called a {\it M\"obius map} if and only if for all admissible quadruples $(w,x,y,z)$ in $Z$, we have $M'(f(w),f(x),f(y),f(z)) = M(w,x,y,z)$. A M\"obius map that is bijective is called a {\it M\"obius equivalence}.

In the same work, Buyalo also showed that for every non-degenerate triple $(\omega, \alpha, \beta) \in Z^3$ in a M\"obius space $(Z,M)$, we can construct a semi-metric $\rho_A$ that induces $M$. Furthermore, for every semi-metric $\rho$ that induces $M$, the following formula holds:
\[ \rho_A(x,y) = \frac{ \rho(x,y) }{Ê\rho(x,\omega) \rho(\omega, y) } \frac{ \rho(\alpha, \omega) \rho(\omega, \beta)}{\rho(\alpha, \beta)}. \]
He used these semi-metrics to define a topology on $(Z, M)$ as follows. Given a triple $A$, we denote the `open ball' at $z$ of radius $r$ with respect to $\rho_A$ by
\[ B_{A,r}(z) := \{ z' \in Z \vert \rho_A(z,z') < r \}. \]
We call the topology generated by the family of open balls for all $A$, $z$, $r$ the M\"obius topology and denote it by $\mathcal{T}_M$. If there exists a metric $\rho$ such that $M$ is induced by $\rho$, then the metric topology $\mathcal{T}_{\rho}$ coincides with $\mathcal{T}_M$ (see \cite{IM17a}).\\

In the following sections, we will only consider M\"obius structures that are induced by quasi-metrics. These can be characterized in terms of the following property.

\begin{mydef} \label{thm:cornercondition}
Let $M : \mathcal{A} \rightarrow \overline{L_4}$ be a M\"obius structure and $crt : \mathcal{A} \rightarrow \overline{\Delta}$ its cross ratio-triple. We say that $M$ satisfies the {\it (corner)-condition} if and only if the closure of $\Ima(crt) \subset \mathbb{R}P^2$ does not contain any of the points $(0:1:1)$, $(1:0:1)$, $(1:1:0)$.

M\"obius structures that satisfy the (corner)-condition are called {\it strong M\"obius structures}.
\end{mydef}

The name of this condition becomes clear when we represent every point in $\overline{\Delta}$ by its unique representative $(\alpha : \beta : \gamma)$ such that $\alpha + \beta + \gamma = 1$. The set $\overline{\Delta}$ then becomes an open triangle in $\mathbb{R}^3$ together with the points $(1:0:0)$, $(0:1:0)$, $(0:0:1)$. The (corner)-condition requires that the image of $crt$ in this triangle does not have any of the corners of the triangle as an accumulation point. We have the following result.

\begin{thm}[\cite{IM17a}] \label{thm:cornercondition}
Let $(Z,M)$ be a M\"obius space. Then $M$ is induced by a (possibly extended) quasi-metric if and only if it satisfies the (corner)-condition. Furthermore, whenever $M$ is induced by a (possibly extended) quasi-metric, there exists a bounded quasi-metric inducing $M$.
\end{thm}

We end this section with a construction that allows us to construct different semi-metrics that induce the same M\"obius structure. Let $\rho$ be an extended semi-metric and let $o \in Z$ be a point such that for all $z \neq o$, $\rho(z,o) > 0$. We define the {\it involution of $\rho$ at $o$} by
\[ \rho_o(x,y) := \frac{ \rho(x,y)Ê}{ \rho(x,o) \rho(o,y) }. \]
Note that $o$ lies at infinity with respect to $\rho_o$ and, if $\omega$ is a point at infinity with respect to $\rho$, then
\[ \rho_o(x,\omega) = \frac{1}{\rho(x,o)}.Ê\]

If $\rho$ is an extended semi-metric, then $\rho_o$ is again an extended semi-metric. In \cite{BS}, Proposition 5.3.6, Buyalo-Schroeder prove that for any extended $K$-quasi-metric $\rho$, its involution $\rho_o$ is a $K'^2$-quasi-metric for some $K' \geq K$. A direct computation shows that $\rho$ and $\rho_o$ induce the same cross ratio. Furthermore, if $\rho$ has a point at infinity, denoted $\omega$, then the involution of $\rho_o$ at $\omega$ is again $\rho$, i.e.
\[ \rho(x,y) = \frac{\rho_o(x,y)}{\rho_o(x,\omega) \rho_o(\omega,y)}. \]




\section{The Hausdorff-dimension of strong M\"obius spaces} \label{sec:HausdorffDimension}

In this section, we generalise the notion of Hausdorff-dimension to strong M\"obius spaces. Consider a  strong M\"obius space $(Z,M)$ and let $\rho$ be a (possibly extended) $K$-quasi-metric that induces $M$. Let $A \subseteq Z$ be a subset, $s \geq 0$ and $\delta > 0$. We denote by $B_{\rho, r}(x) := \{ y \in Z \vert \rho(x,y) \leq r \}$ the closed ball of radius $r$ around $x$ with respect to $\rho$. Furthermore, we say a {\it $\delta$-cover} of $A$ is a cover of $A$ by closed balls $B_{\rho, r_i}(x_i)$ such that for all $i \in I$, $r_i \leq \delta$ (note that we do not require $x_i \in A$). To avoid clustered notation, we will often omit the index set $I$ and simply speak of the index $i$. We define
\[ \mu_{\delta,\rho}^s(A) := \inf \left( \sum_{i} r_i^s \Big|  \{ B_{\rho, r_i}(x_i) \}_{i} \text{ is a $\delta$-cover of $A\setminus \{ \infty \}$}  \right). \]

Fixing $A \subseteq Z$, $\mu_{\delta,\rho}^s(A)$ is increasing as $\delta$ goes to zero. We define the {\it $s$-dimensional Hausdorff measure} of $A$ to be the limit
\[ \mu_\rho^s(A) := \lim_{\delta \rightarrow 0}\mu_{\delta}^s(A) \in [0, \infty]. \]

By construction, we have that $\mu_\rho^s(\emptyset) = 0$ and $\mu_\rho^s$ is monotone and subadditive for countable unions of subsets. Thus, we have defined an outer measure on $Z$. This definition is completely analogous to the definition of the $s$-dimensional Hausdorff measure on a metric space, except that we have to exclude the point at infinity, if it belongs to $A$. In preparation to dealing with the point at infinity, we first show that adding or removing one point that does not lie at infinity from the set $A$, does not change $\mu_{\delta,\rho}^s(A)$.

Let $A \subseteq Z$ and $p \in Z \diagdown (A \cup \{ \infty \})$. Let $\{ B_{\rho, r_i}(x_i) \}_i$ be a $\delta$-cover of $A$. By adding one ball $B_{\delta}(p)$ of radius $\delta$, we get a cover of $A \cup \{ p \}$. When computing $\mu_{\delta,\rho}^s(A \cup \{ p \})$, this cover contributes with the expression
\[ \sum_{i} r_i^s + \delta^s. \]
Taking the infimum over all possible $\delta$-covers, we see that
\[ \mu_{\delta,\rho}^s(A) \leq \mu_{\delta,\rho}^s(A \cup \{ p \}) \leq \mu_{\delta,\rho}^s(A) + \delta^s. \]
Taking the limit $\delta \rightarrow 0$ yields
\[ \mu_\rho^s(A) = \mu_\rho^s(A \cup \{ p \}). \]
So we see that $\mu_\rho^s$ doesn't change when we add or remove a point in $Z$. We will later use this to add a point at infinity when necessary. We now prove the following lemma.

\begin{lem} \label{lem:CriticalValue}
Let $\rho$ be a quasi-metric on $Z$ and $s \geq 0$. Suppose, $\mu_\rho^s(Z) < \infty$. Then for all $t > s, \mu_\rho^t(Z) = 0$.
\end{lem}

\begin{proof}
The proof is exactly the same as for the Hausdorff measure of a metric space. Let $\{ B_{\rho, r_i}(x_i) \}_i$ be a $\delta$-cover of $Z \diagdown \{ \infty \}$. Then
\begin{equation*}
\begin{split}
\sum_i r_i^t & = \sum_i r_i^s r_i^{t-s} \leq \sum_i r_i^s \delta^{t-s} = \delta^{t-s} \sum_i r_i^s.
\end{split}
\end{equation*}
Taking the infimum over all $\delta$-covers, we get
\[ \mu_{\delta,\rho}^t(Z) \leq \delta^{t-s}\mu_{\delta,\rho}^s(Z). \]

By assumption, $\mu_\rho^s(Z) < \infty$ and by construction, $\mu_{\delta,\rho}^s(Z) \leq \mu_\rho^s(Z) < \infty$. Thus, as we let $\delta \rightarrow 0$, we get
\[ \mu_\rho^t(Z) \leq 0. \]
Since $\mu_\rho^t$ is non-negative, this implies that $\mu_\rho^t(Z) = 0$.

\end{proof}

Lemma \ref{lem:CriticalValue} tells us that there is a critical value $c \in [0, \infty]$, such that for all $s < c$, $\mu_\rho^s(Z) = \infty$, while for all $t > c, \mu_\rho^t(Z) = 0$.

\begin{mydef} \label{def:HausdorffDimension}
Let $(Z,M)$ be a strong M\"obius space. Choose a quasi-metric $\rho$ that induces $M$ and define the Hausdorff measures $\mu_\rho^s$ on $Z$ using $\rho$. Define the {\it Hausdorff-dimension} of $(Z,\rho)$ to be $\dim_{Haus}(Z,\rho) := \inf(s \in \mathbb{R} \vert \mu_\rho^s(Z) = 0).$
\end{mydef}

\begin{thmdef} \label{thm:HausdorffInvariance}
Let $(Z, M)$ be a strong M\"obius space. Then, for all (extended) quasi-metrics $\rho$, $\rho'$ that induce $M$, we have $\dim_{Haus}(Z,\rho) = \dim_{Haus}(Z, \rho')$. We define the Hausdorff-dimension of $(Z,M)$ to be
\[Ê\dim_{Haus}(Z, M) := \dim_{Haus}(Z, \rho) \]
for any (extended) quasi-metric that induces $M$.

Furthermore, $\dim_{Haus}(Z, M)$ is invariant under M\"obius equivalence.
\end{thmdef}

Recall from section \ref{sec:Preliminaries} that, if $\rho$ induces $M$, then there is a family of quasi-metrics $\rho_A$ that can be expressed as a rescaling and an involution of $\rho$. Specifically, for $A = (\omega, \alpha, \beta)$, we can write
\[ \rho_A(x,y) = \frac{\rho(x,y)}{\rho(x, \omega)\rho(\omega,y)}\frac{\rho(\alpha, \omega)\rho(\omega, \beta)}{\rho(\alpha, \beta)}. \]
Furthermore, if $\rho$ has a point at infinity, denoted $\omega$, we can take a point $o \in Z$ and consider the involution $\rho_o$. A direct computation shows that the involution of $\rho_o$ at $\omega$ is equal to $\rho$ again. Therefore, taking the involution is a reversible operation (if we have a point at infinity). We see that, if we can show that rescaling and taking an involution of $\rho$ does not change the Hausdorff-dimension, then we have shown that every quasi-metric $\rho$ that induces $M$ induces the same Hausdorff-dimension.\\

\begin{proof}[Proof of Theorem \ref{thm:HausdorffInvariance}]

Suppose $(Z,\rho)$ has no point at infinity. We can then enlarge $Z$ by an additional point, denoted $\infty$, and extend the $K$-quasi-metric $\rho$ to a $K$-quasi-metric on $Z \cup \{ \infty \}$ by setting $\rho(x, \infty) := \infty$ for all $x \in Z$ and $\rho(\infty, \infty) := 0$. This yields an extended $K$-quasi-metric space with the same Hausdorff-dimension as $(Z,\rho)$ as the Hausdorff-measure does not change when adding or removing a single point. Thus, we can assume without loss of generality that $(Z,\rho)$ has a point at infinity.

We start by looking at rescalings of $\rho$. Let $\lambda > 0$ and consider the quasi-metrics $\rho$ and $\lambda \rho$. We see immediately from the definition that $\mu_{\lambda \rho}^s = \lambda^s \mu_{\rho}^s$ for all $\lambda$, $s > 0$. Therefore, the Hausdorff-dimension does not change.

Now let $\rho'$ be the involution of $\rho$ at the point $o \in Z$. We need the following lemma.

\begin{lem} \label{lem:AuxiliaryHausdorff}
Fix a constant $\epsilon > 0$. For all $\delta < \frac{\epsilon}{K^2}$ and all $\delta$-covers $\{ B_{\rho,r_i}(x_i) \}_i$ of $Z \diagdown \{ \infty \}$, there is a subfamily of the collection $\{ B_{\rho', \frac{K^3}{\epsilon^2}r_i}(x_i) \}_i$ that is a $\frac{K^3}{\epsilon^2}\delta$-cover of $Z \diagdown ( \{ \infty \} \cup B_{\rho, \epsilon}(o) )$.
\end{lem}

We will first prove how the lemma implies invariance under involution. Let $\epsilon = \frac{1}{n}$. Let $\{ B_{\rho,r_i}(x_i)Ê\}_i$ be a $\delta$-covering of $Z \diagdown \{ \infty \}$ for $\delta < \frac{1}{K^2n^3} < \frac{\epsilon}{K^2}$. Then we find a $n^2K^3\delta$-covering of $Z \diagdown ( \{ \infty \} \cup B_{\rho,\epsilon}(o) )$ of the form $B_{\rho', n^2K^3r_i}(x_i)$ where $i$ runs over a subset of the indices of the original $\delta$-covering.
Thus,
\[ \mu_{n^2K^3\delta,\rho'}^s( Z \diagdown ( \{ \infty \} \cup B_{\rho, \frac{1}{n}}(o) )) \leq \sum_i \left( n^2K^3 r_i \right)^s = \left( n^2K^3 \right)^s \sum_i r_i^s. \]
Taking the infimum over all $\delta$-coverings of $Z \diagdown \{ \infty \}$, we get
\[ \mu_{n^2K^3\delta, \rho'}^s(Z \diagdown ( \{ \infty \} \cup B_{\rho, \frac{1}{n}}(o) ) ) \leq \left( n^2K^3 \right)^s \mu_{\delta,\rho}^s(Z). \]

Now suppose, $\mu_{\rho}^s(Z) = 0$. Then $\mu_{\delta, \rho}^s(Z) = 0$ for all $\delta > 0$ and hence $\mu_{n^2K^3\delta, \rho'}^s(Z \diagdown ( \{ \infty \} \cup B_{\rho, \frac{1}{n}}(o)) ) = 0$ for all $n \in \mathbb{N}$ and all $0 < \delta < \frac{1}{n^3 K^2}$ (and thus for all $\delta > 0$).

We have seen earlier that adding one point doesn't change the Hausdorff measure. Therefore, we can add the point $\infty$ and get
\[ \mu_{n^2K^3\delta, \rho'}^s(Z \diagdown B_{\rho, \frac{1}{n}}(o)) = 0 \]
for all $0 < \delta < \frac{1}{n^3 K^2}$ and $n \in \mathbb{N}$. Taking the limit for $\delta \rightarrow 0$ yields
\[ \mu_{\rho'}^s(Z \diagdown B_{\rho, \frac{1}{n}}(o)) = 0 \]
for all $n \in \mathbb{N}$. Using the $\sigma$-subadditivity of $\mu_{\rho'}^s$, we see that
\[ \mu_{\rho'}^s(Z) = \mu_{\rho'}^s(Z \diagdown \{ o \}) \leq \sum_{n=1}^{\infty} \mu_{\rho'}^s (Z \diagdown B_{\rho, \frac{1}{n}}(o)) = \sum_{n=1}^{\infty} 0 = 0. \]
We conclude that if $\mu_{\rho}^s(Z) = 0$, then $\mu_{\rho'}^s(Z) = 0$. Therefore, the Hausdorff-dimension of $(Z,\rho')$ is at most the Hausdorff-dimension of $(Z,\rho)$. Since the involution of $\rho'$ at the point $\infty$ is $\rho$ again, we get the reversed inequality as well. This implies that the Hausdorff-dimension is invariant under involution and that the Hausdorff-dimension of a strong M\"obius space is well-defined.\\

We are left to prove that the Hausdorff-dimension is invariant under M\"obius equivalence. Let $f : (Z,M) \rightarrow (Z',M')$ be a M\"obius equivalence. Consider the quasi-metric $\rho_A$ for some triple of mutually different points $A \in Z^3$. Since $f$ is a bijection, $f(A)$ is a triple of mutually different points in $Z'$. Since $f$ is a M\"obius equivalence and $\rho_A$ can be expressed purely in terms of $M$, we have $\rho_A(x,y) = \rho_{f(A)}(f(x),f(y))$ (cf.\,\cite{Buyalo} or \cite{IM17a}). 
Therefore, $\dim_{Haus}(Z, \rho_A) = \dim_{Haus}(Z', \rho_{f(A)})$ and $\dim_{Haus}(Z,M) = \dim_{Haus}(Z,M')$. This implies Theorem \ref{thm:HausdorffInvariance} (up to the proof of Lemma \ref{lem:AuxiliaryHausdorff}).
\end{proof}

\begin{proof}[Proof of Lemma \ref{lem:AuxiliaryHausdorff}]
Let $\{ B_{\rho,r_i}(x_i)Ê\}_i$ be a $\delta$-cover of $Z \diagdown \{ \infty \}$. Recall that
\[ \rho'(x_i,y) = \frac{\rho(x_i,y)}{\rho(x_i,Êo) \rho(o, y)} \leq \frac{r_i}{\rho(x_i, o) \rho(o, y)}. \]

Let $B_{\rho,r_i}(x_i)$ be a ball such that $\rho(x_i, o) > \frac{\epsilon}{K}$. We claim that the collection of balls $B_{\rho, r_i}(x_i)$ with such $x_i$ covers $Z \diagdown ( \{ \infty \} \cup B_{\rho, \epsilon}(o))$. Let $y \in Z \diagdown ( \{ \infty \} \cup B_{\rho, \epsilon}(o) )$. We find some $x_i$ such that $y \in B_{\rho, r_i}(x_i)$. We see that
\[ \epsilon < \rho(y, o) \leq K \max(\rho(y, x_i), \rho(x_i, o)). \]

Since $K\rho(y, x_i) \leq K\delta < \frac{\epsilon}{K} < \epsilon$, the inequality above implies $\epsilon < K\rho(x_i, o)$ and thus the collection of $B_{\rho,r_i}(x_i)$ with $\rho(x_i, o) > \frac{\epsilon}{K}$ covers $Z \diagdown ( \{ \infty \} \cup B_{\rho, \epsilon}(o))$. An analogous argument shows that for all $x_i$ with $\rho(x_i, o) > \frac{\epsilon}{K}$ and $y \in B_{\rho, r_i}(x_i)$, we have $\rho(y, o) > \frac{\epsilon}{K^2}$. Thus, we have for all $x_i$ with $\rho(x_i, o) > \frac{\epsilon}{K}$ and for all $y \in B_{\rho, r_i}(x_i)$
\begin{equation*}
\begin{split}
\rho'(x_i,y) & \leq \frac{r_i}{\rho(x_i, o) \rho(o, y)}\\
& \leq \frac{K^3r_i}{\epsilon^2}.
\end{split}
\end{equation*}

This implies that $B_{\rho,r_i}(x_i) \subseteq B_{\rho', \frac{K^3r_i}{\epsilon^2}}(x_i)$ and thus the collection $\left\{ B_{\rho', \frac{K^2r_i}{\epsilon^2}}(x_i) \right\}_{\rho(x_i, \infty) > \frac{\epsilon}{K}}$ is a $\frac{K^3\delta}{\epsilon^2}$-covering of $Z \diagdown ( \{ \infty \} \cup B_{\rho, \epsilon}(o) )$. This proves the lemma.

\end{proof}

\begin{rem}
Consider a metric space $(Z, \rho)$ together with a Borel measure $\mu$ and let $Q > 0$. We call $\mu$ {\it Ahlfors $Q$-regular} if there exists a constant $C > 0$ such that for every $R > 0$ and every ball $B_R$ of radius $R$, we have
\[ \frac{1}{C}R^Q \leq \mu(B_R) \leq CR^Q.Ê\]

It is a well-known result that, whenever $(Z,\rho)$ admits an Ahlfors $Q$-regular measure, the Hausdorff-dimension of $(Z,\rho)$ is equal to $Q$ (cf.\,\cite{Heinonen}). The same is true for quasi-metric spaces $(Z,\rho)$ and the notion of Hausdorff-dimension for quasi-metric spaces we introduced above. The proof is analogous to the proof in the metric case.

In \cite{LiShanmugalingam}, Li and Shanmugalingam showed that, whenever a metric space $(Z,\rho)$ admits an Ahlfors $Q$-regular measure, there is a metric $\rho'$ which is bi-Lipschitz to the quasi-metric obtained by taking the involution of $\rho$ at a point $o \in Z$ and $(Z, \rho')$ admits an Ahlfors $Q$-regular measure as well. This proves that the Hausdorff-dimension for metric spaces is invariant under the operation of taking the involution of a metric and then taking a specific metric that is bi-Lipschitz to the involution.

It is possible that this approach generalizes to quasi-metric spaces, yielding a quasi-metric notion of Ahlfors $Q$-regularity for Borel measures on strong M\"obius spaces and an alternativ proof of the invariance of the Hausdorff-dimension under involutions. However, if we want to define Ahlfors $Q$-regularity on strong M\"obius spaces by considering quasi-metric spaces first and then proving invariance under rescaling and involution -- as we have done here -- we are confronted with the fact that the topology induced by one quasi-metric $\rho$ does not agree with the M\"obius topology of $M_\rho$ in general (the quasi-metric on the dyadic numbers described in \cite{Schroederdyadic} provides an example for this). Thus, we run into some difficulties. For example, we have to ask whether the Borel $\sigma$-algebra induced by a quasi-metric $\rho$ is the same as the Borel $\sigma$-algebra induced by the M\"obius topology of a strong M\"obius structure.
\end{rem}




\section{The Nagata-dimension of strong M\"obius spaces} \label{sec:NagataDimension}

The Nagata-dimension has been introduced in a note by Assouad (\cite{Assouad}). It is a variation of the asymptotic dimension and can be defined as follows:

\begin{mydef} \label{mydef:NagataDimension}
Let $(Z,\rho)$ be a metric space and $\mathcal{B}$ a cover of $Z$. The cover $\mathcal{B}$ is called {\it $C$-bounded} if for every $B \in \mathcal{B}$, diam$(B) \leq C$.

Let $s > 0$, $m \in \mathbb{N}$. We say that a family of subsets $\mathcal{B} \subset \mathcal{P}(Z)$ has {\it $s$-multiplicity $\leq m$} if for every set $U \in Z$ with diam$(U) \leq s$, there are at most $m$ elements $B \in \mathcal{B}$ with $B \cap U \neq \emptyset$.

The Nagata-dimension $\dim_N(Z,\rho)$ is defined to be the infimum of all $n$ such that there exists a constant $c > 0$ such that for all $s > 0$ there exists a $cs$-bounded cover of $Z$ with $s$-multiplicity $\leq n+1$.
\end{mydef}

The class of metric spaces that have finite Nagata-dimension includes doubling spaces, metric trees, euclidean buildings, homogeneous or pinched negatively curved Hadamard manifolds and others (cf. \cite{LangSchlichenmaierNagata}). The Nagata-dimension is not preserved by quasi-isometries.

In order to generalise the notion of Nagata-dimension to strong M\"obius spaces, we define it for extended quasi-metrics.
\begin{mydef}
Let $(Z, \rho)$ be an (extended) quasi-metric space, $\mathcal{B}$ be a cover of $Z \diagdown \{ \infty \}$ and $C > 0$. We say that $\mathcal{B}$ is a {\it $C$-bounded cover of $Z$}, if every set $B \in \mathcal{B}$ has diam$(B) \leq C$.

Let $s > 0$, $m \in \mathbb{N}$ and $\mathcal{B} \subset \mathcal{P}(Z)$ be a collection of subsets of $Z$. We say that {\it $\mathcal{B}$ has $s$-multiplicity $\leq m$} if every set $U \subset Z$ with diam$(U) \leq s$ intersects at most $m$ elements of $\mathcal{B}$.

The Nagata-dimension $\dim_N(Z,\rho)$ is defined to be the infimum of all $n$ for which there exists a constant $c > 0$ such that for all $s > 0$ there exists a $cs$-bounded cover of $Z$ with $s$-multiplicity $\leq n+1$.
\end{mydef}

\begin{thmdef} \label{thm:NagataInvariance}
Let $(Z,\rho)$ be an (extended) quasi-metric space, $o \in Z$. Let $\rho'$ be a rescaling or an involution of $\rho$. Then $\dim_N(Z,\rho) = \dim_N(Z,\rho')$.

Therefore, given a strong M\"obius structure $M$ and any two quasi-metrics $\rho, \rho'$ that induce $M$, we have $\dim_N(Z,\rho) = \dim_N(Z,\rho')$. Thus we can define $\dim_N(Z,M) := \dim_N(Z,\rho)$ for any quasi-metric $\rho$ that induces $M$. Furthermore, $\dim_N(Z,M)$ is invariant under M\"obius equivalence.
\end{thmdef}

Note that the Nagata-dimension does not depend at all on whether $(Z,\rho)$ has a point at infinity or not. This doesn't pose a problem due to the following proposition, which is a generalization of Proposition 2.7 in \cite{LangSchlichenmaierNagata} and \cite{Lieb}.

\begin{prop} \label{prop:Supremum}
Let $(Z,\rho)$ be a quasi-metric space such that $Z = X \cup Y$. Then $\dim_N(Z,\rho) = \max(\dim_N(X,\rho), \dim_N(Y,\rho))$.
\end{prop}

The Nagata-dimension of a point is zero and hence $\dim_N(Z, \rho) = \dim_N(Z \diagdown \{Êp \},\rho)$ for any point $p \in Z$. This justifies that the Nagata-dimension of an extended quasi-metric space should not change when one removes the point at infinity. The proof of Proposition \ref{prop:Supremum} is a simple generalization of the proof in \cite{LangSchlichenmaierNagata} (cf. also \cite{Lieb}).

\begin{proof}
By definition of the Nagata-dimension, $\dim_N(Z, \rho) \geq \dim_N(Z', \rho)$ for every subspace $Z' \subset Z$. Thus, $\dim_N(Z,\rho) \geq \max(\dim_N(X,\rho), \dim_N(Y,\rho))$. For the other inequality, suppose $\dim_N(X,\rho), \dim_N(Y,\rho) \leq n < \infty$. We find a constant $c > 0$ such that for every $s > 0$, there is a $cs$-bounded cover $\mathcal{D}$ of $Y$ with $s$-multiplicity $\leq n+1$ and a $cK^4(1+c)s$-bounded cover $\mathcal{C}$ of $X$ with $K^5(1+c)s$-multiplicity $\leq n+1$.

We now construct a cover of $Z$. Let $\mathcal{K}$ denote the set of all $D \in \mathcal{D}$ such that there is no pair of points $y \in D, x \in \bigcup_{C \in \mathcal{C}} C$ with $\rho(x,y) \leq s$. Denote by $\mathcal{L} := \mathcal{D} \diagdown \mathcal{K}$. Then, for every $D \in \mathcal{L}$ there exist $C \in \mathcal{C}, x \in C$ and $y \in D$ such that $\rho(x,y) \leq s$. We can thus choose a map $m : \mathcal{L} \rightarrow \mathcal{C} \times X \times Y$ sending $D$ to a triple $m(D) = (C(D), x(D), y(D))$ such that $x(D) \in C(D),$ $y(D) \in D$ and $\rho(x(D),y(D)) \leq s$.

For every $C \in \mathcal{C}$, we define $B_C := C \cup_{C(D) = C} D$. For $D \in \mathcal{K}$, we define $B_D := D$. Since every element of $\mathcal{C}$ and $\mathcal{D}$ appears in the definition of some $B_C$ or $B_D$, the collection $\mathcal{B} := \{ B_C \vert C \in \mathcal{C} \} \cup \{ B_D \vert D \in \mathcal{L} \}$ is a cover of $Z$.

We are left to compute the diameter of $B_C$ and estimate the $s$-multiplicity. By construction and using the fact that $\rho$ is a $K$-quasi-metric, the diameter of $B_C$ has to be bounded by $K^4 \max(cK^4(1+c), c, 1)s = cK^{8}(1+c)s$. Since $B_D$ has diameter at most $cs$, we see that $\mathcal{B}$ is $cK^{8}(1+c)s$-bounded.

For the multiplicity, let $U$ be a subset of $Z$ with diam$(U) \leq s$. If $U \cap C = \emptyset$ for all $C \in \mathcal{C}$, then, by the $s$-multiplicity of $\mathcal{D}$, $U$ meets at most $n+1$ many elements of $\mathcal{D}$ and hence at most $n+1$ many elements of $\mathcal{B}$, since every element of $\mathcal{C}, \mathcal{D}$ appears in exactly one element of $\mathcal{B}$. Now suppose there is some $C \in \mathcal{C}$ with $U \cap C \neq \emptyset$. Then $U \cap D \neq \emptyset$ can only be true for $D \in \mathcal{L}$. In particular, we have some $y(D) \in D$ and some $x(D) \in C(D)$ such that $\rho(x(D),y(D)) \leq s$. Define $U' := U \cup \{ x(D) \vert D \cap U \neq \emptyset \}$. Then diam$(U') \leq K^4\max(1, c)s \leq K^4(1+c)s$. Since the $K^4(1+c)s$-multiplicity of $\mathcal{C}$ is $\leq n+1$, we see that $U'$ can intersect at most $n+1$ many $C \in \mathcal{C}$. Thus, $U$ can intersect at most $n+1$ many sets of $\mathcal{B}$ and we conclude that the $s$-multiplicity of $\mathcal{B}$ is $\leq n+1$.
\end{proof}

As for the Hausdorff-dimension, in order to prove Theorem \ref{thm:NagataInvariance}, we need to show that the Nagata-dimension of quasi-metric spaces is invariant under rescaling and involution. One can see directly from the definition that the Nagata-dimension is invariant under rescaling (in fact, it is invariant under bi-Lipschitz maps). We are left to prove that the Nagata-dimension is invariant under taking involution. This is the content of the following proposition which is a generalization of a result due to Xiangdong (cf.\,Theorem 4.2 and Proposition 4.3 in \cite{XieXiangdong}).

\begin{prop} \label{prop:NagataInvolution}
Let $\rho$ be a $K$-quasi-metric on $Z$, $o \in Z \diagdown \{ \infty \}$. Let $\rho_o$ be the involution of $\rho$ at the point $o$. Then

\[ \dim_N(Z,\rho) = \dim_N(Z,\rho_o). \]
\end{prop}

We claim that, in order to prove proposition \ref{prop:NagataInvolution}, it is enough to prove that $\dim_N(Z,\rho) \geq \dim_N(Z,\rho_o)$. There are two cases to consider. If $Z$ has a point $\infty$ at infinity with respect to $\rho$, then the involution of $\rho_o$ at $\infty$ is again $\rho$. Therefore, the roles of $\rho$ and $\rho_o$ in the lemma above are interchangeable and it is sufficient to prove $\dim_N(Z,\rho) \geq \dim_N(Z,\rho_o)$.

If $Z$ has no point at infinity with respect to $\rho$, we extend $(Z,\rho)$ by adding a point $\infty$ at infinity. When doing so, $\rho$ remains a quasi-metric and we can write $\rho$ as the involution of $\rho_o$ at the point $\infty$ as above. Since adding a single point does not change the Nagata-dimension, the claim now follows from the first case. It is thus sufficient to prove that $\dim_N(Z, \rho) \geq \dim_N(Z,\rho_o)$ in order to prove proposition \ref{prop:NagataInvolution}. In particular, this tells us that we may assume that $\dim_N(Z,\rho) < \infty$.\\

Before we can prove Proposition \ref{prop:NagataInvolution}, we need to do some preparations. Specifically, we need to state and generalize several results from \cite{LangSchlichenmaierNagata}.

\begin{prop} \label{prop:NagataDimensionCharacterization}
Let $(Z,\rho)$ be a quasi-metric space with $K$-quasi-metric $\rho$ and $n \geq 0$ an integer. The following are equivalent:

\begin{enumerate}

\item[(1)] The Nagata-dimension $\dim_N(Z,\rho) \leq n$, i.e. there exists a constant $c_1 > 0$ such that for all $s > 0,$ $Z \setminus \{ \infty \}$ has a $c_1s$-bounded covering with $s$-multiplicity $\leq n+1$.

\item[(2)] There exists a constant $c_2 > 0$ such that for all $s > 0,$ $Z \setminus \{ \infty \}$ admits a $c_2s$-bounded covering of the form $\mathcal{B} = \bigcup_{k=1}^{n+1} \mathcal{B}_k$ where each family $\mathcal{B}_k$ has $s$-multiplicity $\leq 1$.

\end{enumerate}
\end{prop}

We use Proposition \ref{prop:NagataDimensionCharacterization} to prove

\begin{prop} \label{prop:LangSchlichenmaierCovering}
Let $(Z,\rho)$ be a quasi-metric space with $\dim_N(Z) \leq n < \infty$. Then there is a constant $c > 0$ such that for all sufficiently large $r > 1$, there exists a sequence of coverings $\mathcal{B}^j$ of $Z \setminus \{Ê\infty \}$ with $j \in \mathbb{Z}$, satisfying the following properties:

\begin{enumerate}

\item[(i)] For every $j \in \mathbb{Z}$, we can write $\mathcal{B}^j = \bigcup_{k=0}^n \mathcal{B}^j_k$ where each $\mathcal{B}^j_k$ is a $cr^j$-bounded family with $r^j$-multiplicity $\leq 1$.

\item[(ii)] For every $j \in \mathbb{Z}$ and $x \in Z$, there exists a set $C \in \mathcal{B}^j$ that contains the closed ball $B_{\rho, r^j}(x)$.

\item[(iii)] For every $k \in \{ 0, \dots, nÊ\}$ and every bounded set $B \subset Z$, there is a set $C \in \mathcal{B}_k := \bigcup_{j \in \mathbb{Z}} \mathcal{B}^j_k$ such that $B \subset C$.

\item[(iv)] Whenever $B \in \mathcal{B}^i_k$ and $C \in \mathcal{B}^j_k$ for some $k$ and $i < j$, then either $B \subset C$ or $\rho(x,y) > r^i$ for all $x \in B, y \in C$.

\end{enumerate}
\end{prop}

Proposition \ref{prop:NagataDimensionCharacterization} is a generalised version of a larger proposition (Proposition 2.5) in \cite{LangSchlichenmaierNagata} which characterises the definition of the Nagata-dimension by four different conditions. While it would be interesting to check whether the entire proposition can be generalised to quasi-metrics, we only need the characterisation presented here for which we provide a more direct proof. Proposition \ref{prop:LangSchlichenmaierCovering} is a direct generalisation of Proposition 4.1 in \cite{LangSchlichenmaierNagata} and so is its proof presented below.

\begin{proof}[Proof of proposition \ref{prop:NagataDimensionCharacterization}]
The implication $(2) \Rightarrow (1)$ is trivial. For $(1) \Rightarrow (2)$, suppose $\dim_N(Z,\rho) = n$ and let $c > 0$ be the constant in the definition of $\dim_N(Z,\rho)$. Let $s > 0$. We can find a $cK^{2n}s$-bounded covering $\mathcal{B}$ of $Z \diagdown \{ \infty \}$ with $K^{2n}s$-multiplicity $\leq n+1$. For simplicity, we assume without loss of generality that $c \geq 1$. We are going to construct $n+1$ families $\mathcal{B}_i$ of $cK^{4n}s$-bounded subsets of $Z$ such that $\cup_{i=1}^{n+1} \mathcal{B}_i$ is a covering of $Z$ and each $\mathcal{B}_i$ has $s$-multiplicity $\leq 1$, which implies (2) with constant $cK^{4n}$.

Let $B \in \mathcal{B}$. We inductively define $N^0B := B$ and $N^iB$ as the $s$-neighbourhood of $N^{i-1}B$ for $i > 0$. Note that diam$(N^iB) \leq K^2 \max(s, \diam(N^{i-1}B)) \leq cK^{2n+2i}s$ and the collection of all $N^iB$ -- denoted $N^i\mathcal{B}$ -- has $K^{2(n-i)}s$-multiplicity $\leq n+1$. Both statements are easily proved by induction. We now define our new covering. Let $i \in \{ 1, \dots n + 1 \}$. We define $\mathcal{B}_i$ to be the collection of sets of the following form
\[ A = \bigcap_{j = 1}^i N^{i-1}B_j \setminus \bigcup_{B \notin \{ B_1, \dots, B_i \}} N^iB \]
where $B_1, \dots, B_i \in \mathcal{B}$ are mutually distinct sets.\\

Since $N^i\mathcal{B}$ has $K^{2(n-i)}s$-multiplicity $\leq n+1$, every point $x$ is contained in at most $n+1$ many elements of $N^i\mathcal{B}$ for every $i \geq 0$. We claim that $\bigcup_{i=1}^{n+1} \mathcal{B}_i$ is a covering of $Z$. We will show this by using induction to prove the following claim for every $x \in Z$ and then show that the induction ends at $i = n+1$.

\begin{claim} \label{claim:Induction}
Let $x \in Z$. Then for all $i \geq 1$, either $x$ is contained in an element of $\mathcal{B}_j$ for some $j \leq i$, or there are mutually distinct $B_1, \dots, B_{i+1} \in \mathcal{B}$ such that $x \in \left( \bigcap_{j=1}^i N^{i-1}B_j \right) \cap N^iB_{i+1}$.
\end{claim}

\begin{proof}[Proof of Claim \ref{claim:Induction}]

Let $x \in Z$. We start the induction at $i = 1$. If $x$ is contained in exactly one element $B \in \mathcal{B}$ and is not contained in $N^1C$ for any other $C \in \mathcal{B}$, then $x$ is contained in an element of $\mathcal{B}_1$. Suppose this is not the case. Then $x$ is contained in at least one element $B \in \mathcal{B}$ (since $\mathcal{B}$ is a covering) and at least one element $N^1C \in N^1\mathcal{B}$ such that $B \neq C$ (since $x$ is not covered by $\mathcal{B}_1$). This concludes the start of the induction.

Suppose now, the claim is true for all $1 \leq j \leq i$ for a fixed $i \geq 1$. We want to prove the claim for $i+1$. Suppose $x$ is not contained in $\mathcal{B}_j$ for all $1 \leq j \leq i$. Then, by induction-assumption, there exist mutually distinct $B_1, \dots B_{i+1}$ such that $x \in \bigcap_{j=1}^i N^{i-1}B_j \cap N^iB_{i+1} \subset \bigcap_{j=1}^{i+1} N^i B_j$. Suppose, $x$ is not contained in any element of $\mathcal{B}_{i+1}$; in particular, $x \notin \bigcap_{j=1}^{i+1} N^iB_j \setminus \bigcup_{B \notin \{ B_1, \dots, B_{i+1} \}} N^{i+1}B$. Therefore, there exists some $C \in \mathcal{B} \setminus \{ B_1, \dots, B_{i+1} \}$, such that $x \in N^{i+1}C$. Hence, $x \in \bigcap_{j=1}^{i+1} N^i B_j \cap N^{i+1}B_{i+2}$ where we denote $B_{i+2} := C$. This implies that either, $x$ is contained in an element of $\mathcal{B}_{j}$ for $j \leq i+1$, or we find mutually distinct elements $B_1, \dots B_{i+2} \in \mathcal{B}$ such that $x \in \bigcap_{j=1}^{i+1} N^iB_j \cap N^{i+1}B_{i+2}$. This proves the claim.
\end{proof}

Let $x \in Z$ and suppose $x$ is not contained in $\mathcal{B}_{j}$ for all $1 \leq j \leq n+1$. By the claim, we find distinct elements $B_1, \dots B_{n+2} \in \mathcal{B}$ such that $x \in \bigcap_{j=1}^{n+2} N^{n+1}B_j$. Since $N^{n+1}Ê\mathcal{B}$ has $K^{-2}s$-multiplicity $\leq n+1$, this cannot happen. Therefore, $x$ has to be contained in $\mathcal{B}_i$ for some $i \leq n+1$. This implies that $\bigcup_{i=1}^{n+1} \mathcal{B}_i$ is a covering of $Z$.\\

We are left to show that $\mathcal{B}_i$ is $cK^{4n}s$-bounded and has $s$-multiplicity $\leq 1$. Let $A \in \mathcal{B}_i$. Then $A \subset N^{i-1}B$ for some $B \in \mathcal{B}$ and therefore, diam$(A) \leq cK^{2n+2(i-1)}s \leq cK^{4n}s$ for all $i \leq n+1$.

For the $s$-multiplicity, consider a subset $U$ with diam$(U) \leq s$. Suppose, $A, \overline{A} \in \mathcal{B}_i$ such that $U$ intersects both $A$ and $\overline{A}$. Then we find $a \in U \cap A$ and $\overline{a} \in U \cap \overline{A}$. Note that $\rho(a, \overline{a}) \leq s$. By construction, $A$ and $\overline{A}$ have the form $\bigcap_{j=1}^i N^{i-1}B_j \setminus \bigcup_{B \notin \{ B_1, \dots, B_i \}} N^i B$ and $\bigcap_{j=1}^i N^{i-1}\overline{B_j} \setminus \bigcup_{B \notin \{ \overline{B_1}, \dots, \overline{B_i} \}} N^i B$ respectively. We want to show that $\{ B_1, \dots B_i \} = \{ \overline{B_1}, \dots, \overline{B_i} \}$ and thus $A = \overline{A}$. Suppose this is not true. Then we find some $l$ such that $\overline{B_l} \notin \{ B_1, \dots B_i \}$. Therefore, $a \notin N^i\overline{B_l}$. However, since $\overline{a} \in N^{i-1}\overline{B_l}$ and $\rho(a, \overline{a}) \leq s$, we obtain $a \in N^i\overline{B_l}$, a contradiction. Therefore, $\{ B_1, \dots, B_i \} = \{ \overline{B_1}, \dots, \overline{B_i} \}$ and $A = \overline{A}$. This implies that $\mathcal{B}_i$ has $s$-multiplicity $\leq 1$ for all $1 \leq i \leq n+1$ which completes the proof of Proposition \ref{prop:NagataDimensionCharacterization}.
\end{proof}

\begin{proof}[Proof of proposition \ref{prop:LangSchlichenmaierCovering}]

By proposition \ref{prop:NagataDimensionCharacterization} we can find $cK^4r^j$-bounded coverings $\mathcal{B}^j = \bigcup_{k=0}^n B_k^j,$ $j \in \mathbb{Z}$ such that $(i)$ holds for each of the coverings, where $c$ is the constant $c_2$ from Proposition \ref{prop:NagataDimensionCharacterization} and $r^j$ in $(i)$ is replaced by $K^4r^j$. Assume without loss of generality $c \geq 1$. Fix a basepoint $z \in Z$. By permuting the $B_k^j$ for fixed $j$ we can assume without loss of generality that for all $m \in \mathbb{Z}$ and $k \in \{ 0, \dots, n \}$ there is a $C \in \mathcal{B}_k^{m(n+1) + k}$ that contains $z$. Replacing each element $C \in \mathcal{B}_k^j$ by $\bigcup_{x \in C} B_{r^j}(x)$ then yields $(ii)$ and $(iii)$. Note that each $\mathcal{B}_k^j$ is now $K^2 \max(cK^4r^j, r^j)$-bounded and has $K^2r^j$-multiplicity $\leq 1$.

We now further modify the $\mathcal{B}_k^j$ to also get property $(iv)$. We write $C \succ B$ if $B \in \mathcal{B}_k^i$, $C \in \mathcal{B}_k^j$, $i < j$ such that there is a pair of points $x \in B, y \in C$ with $\rho(x,y) \leq K^2r^i$. For $C \in \mathcal{B}_k^j$, we denote by $\hat{C}$ the union of $C$ with all $B$'s for which there exists a chain
\[ C \succ C_1 \succ \dots \succ C_m = B, \]
where $m \geq 1$ and $C_l \in \mathcal{B}_k^{j_l}$ for $l \in \{ 1, \dots, m \}$.

\begin{claim} \label{claim:Radius}
Choose $r > K^2$. For all $B$ that appear in the union of $\hat{C}$, there exists some $y \in C$, such that $B \subset B_{R(j_1)}(y)$ where
\[ R(j_1) := cK^8r^{j_1}. \]
\end{claim}

\begin{proof}[Proof of Claim \ref{claim:Radius}]

We use induction over the length $m$ of the chain. If $m = 1$, then clearly $B \subset B_{cK^8r^i}(y)$ for some $y \in C$ and $j_1 = i$. We now do the induction step from $m$ to $m+1$. Given a chain $C \succ C_1 \succ \dots \succ C_{m+1} = B$, we find $x_l, y_l \in C_l, l \in \{ 1, \dots, m+1 \}$ and $y_0 \in C$ such that for all $l \in \{ 1, \dots, m+1 \}$, $\rho(y_{l-1}, x_l) \leq K^2r^{j_l}$ and $\rho(x_l, y_l) \leq cK^6r^{j_l}$. We can then estimate
\begin{equation*}
\begin{split}
\rho(y_0, y_m) & \leq K^2 \max(\rho(y_0,x_1), \rho(x_1, y_1), \rho(y_1, y_m))\\
& \leq K^2 \max(K^2r^{j_1}, cK^6r^{j_1}, cK^8r^{j_2})\\
& \leq K^2 \max(cK^6r^{j_1}, cK^8r^{j_1 - 1})\\
& \leq K^2 \max(cK^6r^{j_1}, cK^6r^{j_1})\\
& \leq cK^8r^{j_1},
\end{split}
\end{equation*}
where we used the induction assumption in the second inequality. Thus, $B \subset B_{cK^8r^{j_1}}(y_0)$ which proves the claim.
\end{proof}

Now choose $r > cK^8$, which implies that $r > K^2$ and $R(j_1) \leq r^{j_1 + 1} \leq r^j$. We compute that diam$(\hat{C}) \leq K^2 \max(R(j_1), cK^6r^j) \leq K^2 \max(r^j, cK^6r^j) = cK^8r^j$ and thus, for all $k$ and $j$ the family of all $\hat{C}$ with $C \in \mathcal{B}_k^j$ is a $cK^8r^j$-bounded family with $r^j$-multiplicity $\leq 1$. (For the multiplicity, note that $\hat{C}$ is contained in the $r^j$-neighbourhood of $C$ and use the same argument as earlier.) Denote these families by $\hat{\mathcal{B}}_k^j$. Note that the $\hat{\mathcal{B}}_k^j$ still satisfy $(ii)$ and $(iii)$.

We now show that the $\hat{\mathcal{B}}_k^j$ also satisfy $(iv)$. Let $B \in \mathcal{B}_k^i, C \in \mathcal{B}_k^j$ for some $k$ and $i < j$ and suppose there are points $x \in \hat{B}, y \in \hat{C}$ such that $\rho(x,y) \leq r^i$. Since $B \in \mathcal{B}_k^i$, we have some $x' \in B$ with $\rho(x, x') \leq r^i$ by Claim \ref{claim:Radius}. By definition of $\hat{C}$ we find a chain $C \succ C_1 \succ \dots \succ C_m$ for some $m \geq 0$ and $C_l \in \mathcal{B}_k^{j_l}$ for all $l \in \{ 1, \dots, m \}$, such that $y \in C_m$. Let $\tilde{l}$ be the largest index with $j_{\tilde{l}} \geq i$, possibly $\tilde{l} = 0$, $j_0 = j$. Since $j > i$ by assumption, such an index exists. Then there is a point $y' \in C_{\tilde{l}}$ such that $\rho(y,y') \leq r^i$, as $j_{\tilde{l}+1} \leq i$. Therefore, $\rho(x',y') \leq K^2 r^i$. Since $\mathcal{B}_k^i$ has $K^2r^i$-multiplicity $\leq 1$, we get that $i < j_{\tilde{l}}$ and $C_{\tilde{l}} \succ B$. Therefore, $\hat{B} \subset \hat{C}$. This implies that the families $\hat{\mathcal{B}}_k^j$ satisfy property $(iv)$ and form coverings of $Z$ with the properties required in Proposition \ref{prop:LangSchlichenmaierCovering}.
\end{proof}

\begin{proof}[Proof of proposition \ref{prop:NagataInvolution}]
The proof of Proposition \ref{prop:NagataInvolution} is a generalization of the proof in \cite{XieXiangdong}. As discussed earlier, we only need to prove that $\dim_N(Z,\rho) \geq \dim_N(Z,\rho_o)$. Moreover, we can assume $n := \dim_N(Z,\rho) < \infty$. We start by noticing that $\dim_N(Z,\rho) = \dim_N(Z \diagdown \{ \infty, o \}, \rho)$. Let $s > 0$ and choose $K \geq 1$ such that $\rho$ and $\rho_o$ are both $K$-quasi-metrics. By Proposition \ref{prop:LangSchlichenmaierCovering}, we find $c \geq 1$, $r > 1$ and a sequence of coverings $\mathcal{B}^j$ of $(Z \diagdown \{ \infty, o \}, \rho)$, $j \in \mathbb{Z}$ with properties (i)-(iv) of Proposition \ref{prop:LangSchlichenmaierCovering}. Put $c'' := K^4c'$, $c' := 2K^3 \tilde{c}^2$, $\tilde{c} := 10crK^4$. 
We will construct a family $\bigcup_{k=0}^n \mathcal{E}_k$ that is $c''s$-bounded with respect to $\rho_o$, covers $Z \diagdown \{ \infty, o \}$ and each $\mathcal{E}_k$ has $s$-multiplicity $\leq 1$. Since adding the two points $\infty, o$ doesn't change the Nagata-dimension, this implies that $\dim_N(Z, \rho_o) \leq n$.\\

Let $a := \inf \{ \rho(x,o) \vert x \in Z \diagdown \{Ê\infty, o \} \}$ and $b :=\sup \{ \rho(x,o) \vert x \in Z \diagdown \{ \infty, o \} \}$. If $0 < a$ and $b < \infty$, then $id : (Z \diagdown \{ \infty, o \}, \rho) \rightarrow (Z \diagdown \{Ê\infty, o \}, \rho_o)$ is bi-Lipschitz and hence preserves the Nagata-dimension. This can be seen by considering the two inequalities
\begin{equation*}
\begin{split}
& \rho_o(x,y) = \frac{\rho(x,y)}{\rho(x,o)\rho(o,y)} \leq \frac{\rho(x,y)}{a^2},\\
& \rho(x,y) = \rho_o(x,y) \rho(x,o) \rho(o,y) \leq \rho_o(x,y) b^2.
\end{split}
\end{equation*}

Thus, we assume from now on that either $a = 0$ or $b = \infty$. Suppose $a > 0$ and therefore, $b = \infty$. Then
\[ \rho_o(x,y) = \frac{\rho(x,y)}{\rho(x,o)\rho(o,y)} \leq \frac{K}{\min(\rho(x,o),\rho(o,y))} \leq \frac{K}{a} \]
for all $x,y \in Z \diagdown \{ \infty, o \}$. Thus, $(Z \diagdown \{ \infty, o \}, \rho_o)$ is bounded by $\frac{K}{a}$. If $c's \geq \frac{K}{a}$, then we can cover $Z \diagdown \{ \infty, o \}$ by one $c's$-bounded set which we then choose as our covering.

Now suppose that $0 < s < \frac{K}{ac'}$ and $a \geq 0$. Consider the set
\[ A_s := \{ x \in Z \vert 0 < \rho(x,o) \leq \frac{2K}{sc'} \}. \]
Since $a < \frac{K}{sc'}$, $A_s$ is non-empty. Note that the complement of $A_s$ is $c's$-bounded with respect to $\rho_o$, since for all $x,y \in Z \diagdown A_s$
\[ \rho_o(x,y) \leq \frac{K}{\min(\rho(x,o),\rho(o,y))} \leq \frac{Kc's}{2K} < c's. \]

Let $x \in A_s$. By property (ii), for each $j \in \mathbb{Z}$ there exists a $C_x^j \in \mathcal{B}^j$ such that $B_{\rho,r^j}(x) \subset C_x^j$. Since $\mathcal{B}^j$ is $cr^j$-bounded, diam$_\rho(C_x^j) \leq cr^j$. We analyze the diameter of $C_x^j$ with respect to $\rho_o$. For any $y,z \in C_x^j$, we have
\[ \rho_o(y,z) = \frac{\rho(y,z)}{\rho(y,o)\rho(o,z)} \leq \frac{cr^j}{\rho(y,o)\rho(o,z)}. \]

We now show that $\diam_{\rho_o}(C_x^j) \rightarrow 0$ as $j \rightarrow - \infty$. For this, we estimate $\rho(y,o)$ from below for all $y \in C_x^j$. We know that
\[ \rho(x,o) \leq K \max(\rho(x,y),\rho(y,o)) \leq K \max(cr^j, \rho(y,o)). \]
Since $x \neq o$, this inequality takes the form
\[ \rho(x,o) \leq K \rho(y,o) \]
for $j$ sufficiently small. Therefore, we see that
\[ \diam_{\rho_o}(C_x^j) \leq \frac{K^2cr^j}{\rho(x,o)^2} \xrightarrow{j \rightarrow -\infty} 0. \]
In particular, this means that for any $x \in A_s$, we find $j$ sufficiently small and a set $C_x^j \in \mathcal{B}^j$ such that $B_{\rho, r^j}(x) \subset C_x^j$ and $\diam_{\rho_o}(C_x^j) \leq \tilde{c}s$.

\begin{lem} \label{lem:lemmaj}
Let $x \in A_s$. Define
\[ j(x) := \sup\{ j \in \mathbb{Z} \vert \text{ there is a } C \in \mathcal{B}^j : B_{\rho,r^j}(x) \subset C \text{ and } \diam_{\rho_o}(C) \leq \tilde{c}s \}. \]
Then $j(x) < \infty$.

\end{lem}

The map $j(x)$ points us to the largest $C_x^j$ that such that the family $\{ÊC_x^{j(x)} \}_{x \in A_s}$ $\tilde{c}s$-bounded. After removing some redundant elements, this family has $s$-multiplicity $\leq n+1$ which is what we want.

\begin{proof}[Proof of Lemma \ref{lem:lemmaj}]

Suppose $a = 0$. Then we can find a sequence $x_i \in Z \diagdown \{ \infty, o \}$ such that $\rho(x_i,o) \rightarrow 0$. Then, for all $j \in \mathbb{Z}$ with $r^j > \rho(x,o)$, we have that $x_i \in B_{\rho,r^j}(x)$ for sufficiently large $i$. Now we see that
\[ \rho_o(x, x_i) = \frac{\rho(x,x_i)}{\rho(x,o)\rho(o,x_i)} \xrightarrow{i \rightarrow \infty} \infty, \]
since $\rho(x, x_i) = \max( \rho(x, x_i), \rho(x_i, o)) \geq \frac{ \rho(x, o) }{K}$ for $i$ sufficiently large. Thus, $\diam_{\rho_o}(C_x^j) \geq \diam_{\rho_o}(B_{\rho,r^j}(x)) = \infty$ whenever $r^j > \rho(x,o)$ and hence $r^{j(x)} \leq \rho(x,o) < \infty$.\\

Now suppose $a > 0$ and, consequently, $b = \infty$. We find a sequence $x_i \in Z \diagdown \{ \infty, o \}$ such that $\rho(x_i, o) \rightarrow \infty$. We choose and fix $I \in \mathbb{N}$ sufficiently large such that
\[ \frac{\rho(x, x_I)}{\rho(o, x_I)} \geq \frac{\rho(x,x_I)}{K \max(\rho(x, x_I), \rho(x,o))} = \frac{1}{K}. \]
For every $j \in \mathbb{Z}$ such that $r^j > \rho(x, x_I)$, we have $x_I \in B_{\rho,r^j}(x)$ and thus,
\begin{equation*}
\begin{split}
\diam_{\rho_o}(C_x^j) & \geq \diam_{\rho_o}(B_{\rho,r^j}(x)) \geq \rho_o(x, x_I)\\
& = \frac{\rho(x,x_I)}{\rho(x,o)\rho(o,x_I)} \geq \frac{1}{K\rho(x,o)} \geq \frac{1}{K}\frac{c's}{2K} > \tilde{c}s,
\end{split}
\end{equation*}
where we use that $x \in A_s$ for the second-to-last inequality. This implies that $r^{j(x)} \leq \rho(x, x_I) < \infty$. We conclude that, in both cases, $j(x) < \infty$.
\end{proof}

We now define a subfamily of $\mathcal{B}$ by $\mathcal{C} := \{ C_x^{j(x)} \vert x \in A_sÊ\}$. Note that we can write $\mathcal{C} = \bigcup_{k=0}^n \mathcal{C}_k$ where $\mathcal{C}_k := \mathcal{C} \cap \mathcal{B}_k$. Each family $\mathcal{C}_k$ may contain too many sets to have $s$-multiplicity $\leq 1$. The next lemma allows us to throw away those elements of $\mathcal{C}$ that are not needed.

\begin{lem} \label{lem:maximalelement}

For every $C \in \mathcal{C}$ there exists a maximal element $\overline{C}$ of $\mathcal{C}$ such that $C \subset \overline{C}$.

\end{lem}

\begin{proof}

It is enough to show that there is no infinite strictly increasing sequence in $\mathcal{C}$. Suppose we have a sequence $( C_i)_i = ( C_{x_i}^{j(x_i)} )_i$ such that $C_{x_i}^{j(x_i)} \subsetneq C_{x_{i+1}}^{j(x_{i+1})}$. Without loss of generality, we can assume that there is a $k \in \{ 0, \dots, n \}$ such that $C_i \in \mathcal{C}_k$ for all $i$ by passing to a subsequence if necessary. Since $\mathcal{B}_k^j$ has $r^j$-multiplicity $\leq 1$, we have $j(x_i) \neq j(x_{i'})$ for all $i \neq i'$. Therefore, we can assume that $j(x_i) \rightarrow \pm \infty$ by again passing to a subsequence if necessary. If $j(x_i) \rightarrow -\infty$, this would imply that the strictly increasing sequence $C_i$ is contained in a ball of radius $\epsilon > 0$ for arbitrarily small $\epsilon$. This implies that all $C_i$ are contained in a single point which cannot be as the sequence is strictly increasing. Thus, $j(x_i) \rightarrow \infty$ and, since $x_i \in A_s$ for all $i$, $A_s$ is bounded, and $C_i \supset B_{\rho, r^{j(x_i)}}(x_i)$, we conclude that $\bigcup_{i=1}^{\infty} C_i = Z \diagdown \{ \infty , o \}$.

On the other hand, $\diam_{\rho_o}(C_i) \leq \tilde{c}s$ for all $i$ and therefore, $\diam_{\rho_o}(Z \diagdown \{ \infty, oÊ\}) \leq \tilde{c}s$. This implies that $a > 0$ and $b = \infty$, since $(Z, \rho_o)$ is unbounded if $a = 0$. Choose a sequence $(y_j)_j$ in $Z \diagdown \{ \infty, o \}$ such that $\rho(y_j, o) \rightarrow \infty$ and, therefore, $\rho(y_j, x) \rightarrow \infty$ for all $x \in A_s$. Then we have
\begin{equation*}
\begin{split}
\diam_{\rho_o}(Z \diagdown \{ \infty, o \}) & \geq \rho_o(x, y_j)\\
& = \frac{\rho(x, y_j)}{\rho(x,o)\rho(o,y_j)}\\
& \geq \frac{\rho(x, y_j)}{\rho(x,o) K \max(\rho(o, x), \rho(x, y_j))}\\
& \xrightarrow{j \rightarrow \infty} \frac{1}{K\rho(x,o)} \geq \frac{c's}{2K^2} > \tilde{c}s.
\end{split}
\end{equation*}
This implies $\tilde{c}s < \diam_{\rho_o}(Z \diagdown \{ \infty, o \}) \leq \tilde{c}s$ which is a contradiction.
\end{proof}

Define by $\mathcal{D}$ the subfamily of $\mathcal{C}$ consisting of the maximal elements of $\mathcal{C}$ with respect to inclusion. By Lemma \ref{lem:maximalelement}, this is still a covering of $A_s$. Furthermore, define $\mathcal{D}_k := \mathcal{D} \cap \mathcal{B}_k$.\\

We claim that for every $k \in \{ 0, \dots, nÊ\}$, $\mathcal{D}_k$ has $s$-multiplicity $\leq 1$ with respect to $\rho_o$. We prove this by proving the following claim.

\begin{claim} \label{claim:multiplicity}
Let $C_x^{j(x)},$ $C_y^{j(y)} \in \mathcal{D}_k$. Then either $C_x^{j(x)} = C_y^{j(y)}$, or $\rho_o(x', y') > s$ for all $x' \in C_x^{j(x)}$, $y' \in C_y^{j(y)}$.
\end{claim}

\begin{proof}[Proof of Claim \ref{claim:multiplicity}]

Suppose $C_x^{j(x)} \neq C_y^{j(y)}$ and suppose, by contradiction, there are $x' \in C_x^{j(x)}$, $y' \in C_y^{j(y)}$ such that $\rho_o(x',y') \leq s$. Without loss of generality, assume $j(x) \leq j(y)$. Since $C_x^{j(x)}$ and $C_y^{j(y)}$ are distinct and maximal elements of $\mathcal{C} \subset \mathcal{B}$, property (iv) of Proposition \ref{prop:LangSchlichenmaierCovering} implies that $\rho(\tilde{x}, \tilde{y}) \geq r^{j(x)}$ for all $\tilde{x} \in C_x^{j(x)}$, $\tilde{y} \in C_y^{j(y)}$. In particular, $\rho(x',y') \geq r^{j(x)}$.

By definition of $j(x)$, we know that $\diam_{\rho_o}(C_x^{j(x)}) \leq \tilde{c}s$. Let $z \in C_x^{j(x)} \cup \{ y'Ê\}$. Recall that $K$ was chosen such that $\rho_o$ is a $K$-quasi-metric as well. We have
\[ \rho_o(x, z) \leq K\max(\rho_o(x,x'),\rho_o(x',z)) \leq K \max(\tilde{c}s, s) = K \tilde{c}s,Ê\]
since $\tilde{c} = 10crK^4 > 1$. On the other hand, 
\[ \rho_o(x,z) = \frac{\rho(x,z)}{\rho(x,o)\rho(o,z)}. \]
Using this, the definitions of $c'$ and $\tilde{c}$ and the fact that $x \in A_s$, we get
\[ \rho(x,z) \leq K \tilde{c} s \frac{2K}{sc'} \rho(o,z) = \frac{1}{K \tilde{c}} \rho(o,z). \]
This implies that
\begin{equation*}
\begin{split}
\rho(z,o) & \leq K \max(\rho(z,x), \rho(x,o))\\
& = K \max\left( \frac{1}{K \tilde{c}} \rho(o,z) , \rho(x,o) \right)\\
& = K\rho(x,o).
\end{split}
\end{equation*}

By putting $z = x'$ and $z = y'$ respectively, we conclude that $\rho(x',o) \leq K\rho(x,o)$ and $\rho(y',o) \leq K\rho(x,o)$. Since $\mathcal{B}^j$ satisfies property (ii), there is a $C' \in \mathcal{B}^{j(x)+1}$ such that $B_{\rho, r^{j(x)+1}}(x) \subset C'$. Let $w \in C'$. Then, since $\mathcal{B}^j$ is $cr^j$-bounded,
\begin{equation*}
\begin{split}
\rho(x, w) & \leq cr^{j(x) + 1} \leq cr \rho(x',y')\\
& \leq cr K \max( \rho(x',x), \rho(x,y') ) \leq \frac{crK}{K\tilde{c}} \max( \rho(x',o), \rho(o,y') ) \leq \frac{1}{10K} \rho(x,o).
\end{split}
\end{equation*}
Therefore,
\[ \rho(w,o) \leq K\max(\rho(w,x), \rho(x,o)) \leq K\max \left(\frac{1}{10K}\rho(x,o), \rho(x,o) \right) = K\rho(x,o) \]
and
\[ \rho(x,o) \leq K \max(\rho(x,w), \rho(w,o)) \leq K \max \left(\frac{1}{10K}\rho(x,o), \rho(w,o) \right) = K\rho(w,o), \]
as $\rho(x,o) > \frac{1}{10 K}\rho(x,o)$. Together, this implies that $\frac{1}{K}\rho(x,o) \leq \rho(w,o) \leq K\rho(x,o)$. Now we can estimate
\begin{equation*}
\begin{split}
\rho_o(x,w) & = \frac{\rho(x,w)}{\rho(x,o)\rho(o,w)}\\
& \leq \frac{cr^{j(x) + 1}}{\frac{1}{K}\rho(x,o)^2}\\
& \leq K \frac{cr \rho(x',y')}{\rho(x,o)^2}\\
& = K \frac{cr \rho_o(x',y')\rho(x',o)\rho(y',o)}{\rho(x,o)^2}\\
& \leq K \frac{cr s K\rho(x,o)K\rho(x,o)}{\rho(x,o)^2}\\
& = crK^3s,\\
\end{split}
\end{equation*}
where we used that $\rho_o(x',y') \leq s$, $\rho(x',o) \leq K \rho(x,o)$, $\rho(y',o) \leq K\rho(x,o)$. Thus, $\diam_{\rho_o}(C') \leq crK^4s = \tilde{c}s$ and $B_{\rho, r^{j(x) + 1}}(x) \subset C'$ which contradicts the definition of $j(x)$. The claim follows.
\end{proof}

From the claim, we conclude that every $\mathcal{D}_k$ has $s$-multiplicity $\leq 1$. Further, $\bigcup_{k=0}^n \mathcal{D}_k$ is a $\tilde{c}s$-bounded cover of $A_s$. Put $B_s := Z \diagdown (A_s \cup \{ \infty, o \})$. If $B_s = \emptyset$, we are done. If $B_s \neq \emptyset$, we estimate its diameter with respect to $\rho_o$ as follows. For all $x,y \in B_s$, we have $\rho(x,o) > \frac{2K}{c's}$ and thus
\[ \rho_o(x,y) = \frac{\rho(x,y)}{\rho(x,o)\rho(o,y)} \leq \frac{K}{\min(\rho(x,o), \rho(o,y))} \leq K\frac{c's}{2K} < c's. \]
Hence, $\diam_{\rho_o}(B_s) < c's$. Define $\mathcal{D}_0^{far} := \{ÊC \in \mathcal{D}_0 \vert \rho_o(C, B_s) > s \}$ and $\mathcal{D}_0^{close} := \{ C \in \mathcal{D}_0 \vert \rho_o(C, B_s) \leq s \}$ where $\rho_o(A,B) := \inf\{ \rho(a,b) \vert a \in A, b \in B \}$. We now define $E := B_s \cup \bigcup_{C \in \mathcal{D}_0^{close}} C$. The diameter of $E$ is bounded by
\[ \diam_{\rho_o}(E) \leq K^4 \max(s, \tilde{c}s, c's) = K^4c' s = c''s. \]

Defining $\mathcal{E}_0 := \mathcal{D}_0^{far} \cup \{ E \}$ and $\mathcal{E}_k := \mathcal{D}_k$ for $k \in \{Ê1, \dots n \}$, yields a $c''s$-bounded covering of $Z \diagdown \{ \infty, o \}$ with respect to $\rho_o$. Clearly, $\mathcal{E}_0$ still has $s$-multiplicity $\leq 1$ and so does $\mathcal{E}_k$ for every $k \in \{ 1, \dots, n \}$. Thus, we have constructed a $c''s$-bounded covering of $Z \diagdown \{ \infty, o \}$ with $s$-multiplicity $\leq n+1$. This proves that $\dim_N(Z, \rho_o) \leq n = \dim_N(Z,\rho)$ which completes the proof of Proposition \ref{prop:NagataInvolution} and Theorem \ref{thm:NagataInvariance}.
\end{proof}

\begin{rem}
We remark that there are quasi-metric spaces that are not bi-Lipschitz equivalent to a metric space. An example is given in \cite{Schroederdyadic}. For these spaces, the definitions of Hausdorff- and Nagata-dimension provide us with new notions of dimension.
\end{rem}

\bibliographystyle{alpha}
\bibliography{mybib}

\end{document}